\documentclass{article}
\usepackage{amssymb}
\usepackage[cp1251]{inputenc}
\usepackage[english]{babel}
\usepackage{amsmath}
\usepackage{amsthm}
\usepackage{graphics}
\usepackage{epsfig}
\usepackage[all,knot]{xy}
\xyoption{arc}
\usepackage{color}

\newtheorem{theorem}{Theorem}[subsection]
\newtheorem{lemma}[theorem]{Lemma}

\newtheorem{corollary}[theorem]{Corollary}
\newtheorem{definition}[theorem]{Definition}
\newtheorem{remark}[theorem]{Remark}
\newtheorem{example}[theorem]{Example}
\def\cB{{\cal B}}
\def\cD{{\cal D}}
\def\rX{{\mathfrak X}}
\def\rg{\mathfrak g}
\def\R{{\bf R}}
\def\Aut{{\hbox{\bf Aut}\;}}
\def\Out{\hbox{\bf Out}}
\def\Der{\hbox{\bf Der}}
\def\kernel{{\hbox{\bf Ker}\;}}
\def\g{\mathfrak{g}}
\def\h{\mathfrak{h}}
\def\D{\mathcal{D}}

\title{Comparison of categorical characteristic classes of transitive Lie algebroid with Chern-Weil homomorphism}

\author{Mishchenko, A.S. \\(Harbin Institute of Technology, China,\\ 
 Moscow State University, Russia),\\
Li Xiaoyu \\(Harbin Institute of Technology, China)}
\begin{document}
\maketitle

\begin{abstract}
Transitive Lie algebroids have specific properties that allow to look at the transitive
Lie algebroid as an element of the object of a homotopy functor. Roughly speaking each
transitive Lie algebroids can be described as a vector bundle over the tangent bundle of
the manifold which is endowed with additional structures. Therefore transitive Lie
algebroids admits a construction of inverse image generated by a smooth mapping of
smooth manifolds.

Due to to K.Mackenzie (\cite{Mck-2005}) the construction can be managed as a homotopy functor $\mathcal{TLA}_{\rg}$ from category of
smooth manifolds to the transitive Lie algebroids. The functor $\mathcal{TLA}_{\rg}$ associates
with each smooth manifold $M$ the set $\mathcal{TLA}_{\rg}(M)$ of all transitive algebroids with
fixed structural finite dimensional Lie algebra $\rg$. Hence one can construct
(\cite{Mi-2010},\cite{Mi-2011})
a classifying space $\cB_{\rg}$ such that
the family of all transitive  Lie algebroids with fixed Lie algebra $\rg$ over the manifold $M$
has one-to-one correspondence with the family of homotopy classes of continuous maps $[M,\cB_{\rg}]$:
$
\mathcal{TLA}_{\rg}(M)\approx [M,\cB_{\rg}].
$

It allows to describe characteristic classes of transitive Lie algebroids from
the point of view a natural transformation of functors similar to the classical
abstract characteristic classes for vector bundles and to compare them with that
derived from the Chern-Weil homomorphism by J.Kubarski(\cite{Kub-91e}).
As a matter of fact we show that the Chern-Weil homomorphism
does not cover all  characteristic classes from categorical
point of view.

\end{abstract}
\newpage

\section{Basic definitions and functor $\mathcal{TLA}_{\mathfrak{g}}(\bullet)$}
\subsection{Definitions}
\begin{definition}(See \cite{Mck-2005},Definition 3.1.1)
A Lie algebroid $A$ over a smooth manifold $M$ is a vector
bundle $p:A\rightarrow M$ together with a Lie algebra
structure $\{\bullet\}$ on the space of $\Gamma^{\infty}(A;M)$ and
a bundle map $a:A\rightarrow TM $ called the anchor,
such that
\begin{enumerate}\renewcommand{\labelenumi}{$($\roman{enumi}$)$}
\item the induced map
$a:\Gamma(A;M)\rightarrow \Gamma(TM;M)=\rX^{1}(M)$
is a Lie algebra homomorphism
\item for any sections $\alpha, \beta\in\Gamma(A;M)$
and smooth function $f\in C^{\infty}(M)$ we have the Leibniz
identity
$$\{\alpha,f\cdot\beta\}=f\cdot\{\alpha,\beta\}+
a(\alpha)(f)\cdot\beta$$
\end{enumerate}

\end{definition}

We call $A$ a regular Lie algebroid if the rank of $a$ is
locally constant and $A$ a transitive Lie algebroid if $a$
is surjective. The  Lie algebroid homomorphism and isomorphism
is defined in \cite{Mck-2005}. And we often use the Atiyah exact
sequence
$
 0\xrightarrow{} L \xrightarrow{j} A \xrightarrow{a}TM \xrightarrow{} 0
$
 to denote a transitive Lie algebroid. Here $L=\kernel a$
 is called the adjoint bundle. Sometimes we use $(A,M,\{bullet\},a)$
 to note Lie algebroid in order to highlight the bracket.
All transitive Lie algebroids(isomorphic class) and
homomorphisms between them form a category that is fundamental
in our considerations.
\begin{example}(See \cite{Mck-2005}) The followings are
important examples of transitive Lie algebroid.
\begin{enumerate}
  \item Let $M$ be a manifold and let $\g$ be a Lie algebra.
  On $TM\bigoplus(M\times\g)$ define $a:TM\bigoplus(M\times\g)\rightarrow TM $
  by $a:(X,\mu)\mapsto X$. And a bracket
  $$\{(X,\mu),(Y,\nu)\}=([X,Y],X(\nu)-Y(\nu)+[\mu,\nu]).$$
  for $(X,\mu),(Y,\nu)\in\Gamma(TM\bigoplus(M\times\g);M)$.

  Then $TM\bigoplus(M\times\g)$ is a transitive Lie algebroid
  on $M$, called the trivial Lie algebroid on $M$ with
  structural Lie algebra $\g$.

  \item Let $L$ be a Lie algebra bundle on smooth manifold $M$.
  The Lie algebroid $\D_{Der}(L)$ of covariant derivatives
  on $\Gamma^{\infty}(L)$ is a transitive Lie algebroid on $M$.

  \item The Lie algebroid $\D(E)$ of covariant differential
  operators on the space of sections of vector bundle $E$.
\end{enumerate}
\end{example}

As vector space is commutative Lie algebra,
vector bundle $E$ is also commutative Lie algebra bundle.
Thus $\D(E)$ and $\D_{Der}(E)$ are identical in this case.
In the following part of this article we use $\g$ to note
Lie algebra and $\h$ to note  commutative Lie algebra.
All the Lie algebras we consider in this article are finite
dimensional.

\subsection{Functor $ \mathcal{TLA}_{\g}(\bullet) $}
In \cite{Mck-2005}, K. Mackenzie defines pullback of
transitive Lie algebroid over smooth map $f:M'\rightarrow M$.
It means that given a Lie algebra $\g$ there is the functor
$\mathcal{TLA}_{\g}(\bullet)$ such that with any manifold $M$
it assigns the family $\mathcal{TLA}_{\g}(M)$ of all
transitive Lie algebroid with structural Lie algebra $\g$.

\begin{lemma}(See \cite{Mck-2005}, page 248)\label{LAB}
Let $ 0\xrightarrow{} L \xrightarrow{j} A \xrightarrow{a}TM \xrightarrow{} 0$
be a transitive Lie algebroid on smooth manifold $M$.
Then $L$ is a Lie algebra bundle with respect to the braces
structure on $\Gamma(L;M)$  induced from the braces
on $\Gamma(A;M)$.
\end{lemma}

\begin{lemma}(See \cite{Mck-2005}, page 100)\label{SubLie}
Let $A$ be a transitive Lie algebroid on $M$ and let
$U\subset M$ be an open subset. Then the braces
$\{,\}:\Gamma(A;M)\times \Gamma(A;M)\rightarrow \Gamma(A;M)$
restricted to
$\Gamma(A_{U};U)\times \Gamma(A_{U};U)\rightarrow \Gamma(A_{U};U)$
make $A_{U}$ be a Lie algebroid on $U$ called the
restriction of $A$ to $U$.
\end{lemma}

\begin{lemma}(See \cite{Mck-2005}, page 317)\label{Localtri}
Consider a transitive Lie algebroid
$ 0\xrightarrow{} L \xrightarrow{j} A \xrightarrow{a}TM \xrightarrow{} 0$
on $M$ with fixed structural Lie algebra $\g$. Given any open covering
$\{U_{\alpha}\}$ of $M$ by contractible sets,
for arbitrary $\alpha$, there is an Lie algebroid isomorphism
$$
S_{\alpha}:TU_{\alpha}\bigoplus(U_{\alpha}\times \g)\rightarrow A_{U_{\alpha}}
$$
where $TU_{\alpha}\bigoplus(U_{\alpha}\times \g)$ is trivial Lie
algebroid on $U_{\alpha}$.
\end{lemma}

By using Lemma \ref{LAB}, Lemma \ref{SubLie}, Lemma \ref{Localtri} and
the method used in \cite{Hatcher-2005}, we get the following theorem.
\begin{theorem}
Let $M$ and $N$ be smooth manifolds. Given an arbitrary transitive Lie
algebroid $A$ on $N$. Let $f,g:M \rightarrow N$ are homotopic smooth maps. Then the pullback of $A$ over $f$ and $g$ are Lie algebroid isomorphic, that is $f^{!!}A\approx g^{!!}A$.
\end{theorem}

Hence the functor $\mathcal{TLA}_{\g}(\bullet)$ is homotopy functor
for fixed structural Lie algebra $\g$. There exists a classifying space
$\cB_{\g}$ such that $\mathcal{TLA}_{\g}(M)$ has one to one
correspondence with the family of homotopy classes of continuous
maps $[M;\cB_{\g}]$. Here $\cB_{\g}$ is abstract and can be described
in more or less understandable way (see \cite{Mi-2011}).
\newpage
\section{Obstruction}
\subsection{Cohomology}
\begin{definition}(see \cite{Mck-2005}, page 107)
Let $A$ be an arbitrary Lie algebroid on a smooth manifold $M$ and $E$
is a vector bundle on $M$. Let $\D(E)$ be the Lie algebroid of
covariant derivative on $\Gamma^{\infty}(E)$. A representation
of $A$ on $E$ is a
Lie algebroid homomorphism
$$
\rho:A\rightarrow \D(E).
$$
\end{definition}

The cohomology space $\mathcal{H}^{n}(A,\rho,E),n\geq 0$ can be
defined when the representation $\rho$ is given(see \cite{Mck-2005},
page 260). When $A$ is $TM$, we denote the representation
by $\nabla:TM\rightarrow E$. Then there is
$\mathcal{H}^{n}(M,\nabla,E),n\geq 0$. The representation
$\nabla:TM\rightarrow E$ can be regard as a flat connection on
$E$(see \cite{Mck-2005}, page 109, page 186 ). Due to Lemma 1.1.6
and Lemma 1.2.2 in \cite{Kub-91e}, the following theorem holds.

\begin{theorem}
Let $E$ be a vector bundle on smooth manifold $M$ and
$\nabla:TM\rightarrow E$ be a representation of $TM$ on $E$.
Let $f:M'\rightarrow M$ be a smooth map between smooth manifold
$M'$ and $M$. Let $E'=f^{*}E$ be the pullback of vector bundle
over $f$. Then
\begin{enumerate}\renewcommand{\labelenumi}{$($\roman{enumi}$)$}
\item the representation $\nabla$ induces a representation
of $TM'$ on $E'$ noted by $\nabla':TM'\rightarrow\D(E')$.
\item the map $f$ induces a homomorphism between cohomologies
$$
f^{*}:\mathcal{H}^{*}(M,\nabla,E)\rightarrow \mathcal{H}^{*}(M',\nabla',E'),
$$
where
$$\mathcal{H}^{*}(M,\nabla,E)=
\bigoplus\limits_{n=0}^{\infty}\mathcal{H}^{n}(M,\nabla,E),
\quad\mathcal{H}^{*}(M',\nabla',E')=
\bigoplus\limits_{n=0}^{\infty}\mathcal{H}^{n}(M',\nabla',E').$$
\end{enumerate}
\end{theorem}
From fundamental differential geometry, the following theorem holds.
\begin{theorem}
Let $E$ be a commutative Lie algebra bundle with fiber $\h$.
Let $\nabla$ be a flat connection on it. Then $\nabla$  induces
the system of transition functions $\{\varphi_{\alpha\beta}\}$
for $E$ that are locally constant.
Then $E$ can be seen as vector bundle with discrete structural
group $\Aut(\h)_{d}$, and denoted by $E^{\nabla}\rightarrow M$.
Here $\Aut(\h)_{d}$ is the group of all automorphisms of $\h$,
that is $\Aut(\h)$, with discrete topology.
\end{theorem}

\subsection{Obstruction class}
Let $L$ ba a Lie algebra bundle on smooth manifold $M$ with fiber $\g$. There is a commutative diagram(see \cite{Mck-2005}).
$$
\xymatrix{
&0\ar[d]               &0\ar[d]     \\
&ZL\ar[r]^{=}\ar[d]^{i}&ZL\ar[d]^{i}\\
&L\ar[r]^{=}\ar[d]^{ad} &L\ar[d]^{ad}\\
0\ar[r]&\Der(L)\ar[r]^{j}\ar[d]^{\natural^{0}}&\cD_{Der}(L)\ar[r]^{a}\ar[d]^{\natural}&TM\ar[r]\ar[d]^{=}&0\\
0\ar[r]&\Out_{Der}(L)\ar[r]^{\bar j}\ar[d]&\Out\cD_{Der}(L)\ar[r]^{\bar a}\ar[d]&TM\ar[r]\ar[d]&0\\
&0&0&0
}
$$
in which both rows and columns are exact.

Consider a coupling $\Xi:TM\rightarrow \Out\D_{Der}(L)$, that is the curvature tensor
$$
R^{\Xi}:\Lambda^{2} (TM)\rightarrow \Out_{Der}(L)
$$
defined by
$$
R^{\Xi}(X,Y)=[\Xi(X),\Xi(Y)]-\Xi([X,Y])
$$
for $X,Y\in\rX^{1}(M)$ is zero.

There is a lifting $\nabla_{\Xi}:TM\rightarrow \D_{Der}(L)$ of the coupling $\Xi$:
$$
\xymatrix{
L\ar[d]^{ad}\\
\cD_{Der}(L)\ar[d]^{\natural}&TM\ar[l]_{\nabla_{\Xi}}
\ar[d]^{=}\\
\Out\cD_{Der}(L)\ar[d]&TM\ar[d]\ar[l]_{\Xi}\\
0&0
}
$$
in which $\nabla$ is vector bundle map.

Then for curvature tensor
$R^{\nabla_{\Xi}}:\Lambda^{2} (TM)\rightarrow \Der(L)$
defined by $R^{\nabla_{\Xi}}(X,Y)=
[\nabla_{\Xi}(X),\nabla_{\Xi}(Y)]-
\nabla_{\Xi}([X,Y])$, the following diagram is commutative.
$$
\xymatrix{
L\ar[d]^{ad}\\
\Der(L)\ar[d]^{\natural^{0}}&\Lambda^{2}(TM)
\ar[l]_{R^{\nabla_{\Xi}}}\ar@/_.5pc/[ld]_{0}
\\
\Out_{Der}(L)\ar[d]&\\
0&
}
$$

Since vertical column is exact there is a lifting of
$R^{\nabla_{\Xi}}$ that is a bundle map $\Omega:\Lambda^{2} (TM)\rightarrow L$ such that the diagram
\begin{equation}\label{lifting}
\xymatrix{
L\ar[d]^{ad}\\
\Der(L)\ar[d]^{\natural^{0}}&\Lambda^{2}(TM)\ar[l]_{R^{\nabla_{\Xi}}}\ar@/_.5pc/[ld]_{0}
\ar@/_.5pc/[lu]_{\Omega}
\\
\Out_{Der}(L)\ar[d]&\\
0&
}
\end{equation}
is commutative.

Define $d^{\nabla}:\Gamma(\Omega^{n}(M,L);M)\rightarrow \Gamma(\Omega^{n+1}(M,L);M)$ by
$$
\begin{array}{ll}
d^{\nabla}f(X_{1},X_{2},...,X_{n+1})=\sum\limits_{i=1}^{n+1}(-1)^{i+1}\nabla(X_{i})(f(X_{1},X_{2},...,\hat{X_{i}},...,X_{n+1})\\
~~~~~~~~~~~~~~~~~ +\sum\limits_{i<j}(-1)^{i+j}f([X_{i},X_{j}],X_{1},...,\hat{X_{i}},...,\hat{X_{j}},...,X_{n+1})
\end{array}
$$
here $f\in\Gamma(\Omega^{n}(M,L);M)$ and
$X_{1},X_{2},...,X_{n+1}\in\rX^{1}$.

For $\Omega$ in diagram (\ref{lifting}),
$d^{\nabla_{\Xi}}\Omega\in\Omega^{3}(M,ZL)$ and
$d^{\nabla^{ZL}_{\Xi}}(d^{\nabla})=0$ where
$\nabla^{ZL}_{\Xi}$ is induced by $\nabla_{\Xi}$ (see \cite{Mck-2005}).
Then define $Obs(\nabla_{\Xi})=[d^{\nabla_{\Xi}}(\Omega)]
\in\mathcal{H}^{3}(M,\nabla^{ZL}_{\Xi},ZL)$.
The connection $\nabla^{ZL}_{\Xi}$ and
cohomology class $Obs(\nabla_{\Xi})$ depend only on $\Xi$
(see \cite{Mck-2005}, page 273 and Theorem 7.2.12).
Then the class $Obs(\nabla_{\Xi})$ is called the
\emph{obstruction class} of the coupling $\Xi$,
and is denoted by $Obs(\Xi)$.

\begin{theorem}(\textbf{The functorial property})
Let $L$ ba a finite dimensional Lie algebra bundle
on a smooth manifold $M$. Let $M'$ be a smooth manifold
and $f:M'\rightarrow M$ is smooth map. Let $L'=f^{*}L$ be
the pullback of Lie algebra bundle over $f$.
Consider a coupling $\Xi:TM\rightarrow \Out\D_{Der}L$.
Then $\Xi$ induces a coupling $\Xi':TM'\rightarrow \Out\D_{Der}L'$ and $f$ induces a homomorphism
$$
f^{*}:\mathcal{H}^{*}(M,\Xi,ZL)\rightarrow \mathcal{H}^{*}(M',\Xi',ZL').
$$

Further more the obstruction class
$Obs(\Xi')\in\mathcal{H}^{3}(M',\Xi',ZL')$ satisfies the condition
$$
f^{*}(Obs(\Xi))=Obs(\Xi')
$$
\end{theorem}

\begin{definition}
An extension of $TM$ by Lie algebra bundle $L$ is an exact sequence of Lie algebroid over $M$
$$
 0\xrightarrow{} L \xrightarrow{j} A \xrightarrow{a}TM \xrightarrow{} 0.
$$
\end{definition}
\begin{theorem}(see \cite{Mck-2005}, corollary 7.3.9)
Let $L$ be a Lie algebra bundle on $M$. Let  $\Xi:TM\rightarrow \Out\D_{Der}(L)$ be a coupling. Then , if $Obs(\Xi)=0$, there is a Lie algebroid extension
$$
0\xrightarrow{} L \xrightarrow{j} A \xrightarrow{a}TM \xrightarrow{} 0
$$
of $TM$ by $L$ inducing the coupling $\Xi$.
\end{theorem}
\begin{corollary}
Let $E$ be a vector bundle over $M$ (that is the Lie algebra bundle with
commutative Lie algebra). There is a Lie algebroid extension
$$
0\xrightarrow{} E \xrightarrow{j} A \xrightarrow{a}TM \xrightarrow{} 0
$$
if and only if the bundle $E$ is flat.
\end{corollary}
\begin{proof}
Suppose that the extension exists
$$
0\xrightarrow{} E \xrightarrow{j} A \xrightarrow{a}TM \xrightarrow{} 0
$$
Let $\lambda:TM\rightarrow A$ be a splitting. Define
$$
\nabla^{\lambda}:\rX^{1}(M)\times\Gamma^{\infty}(E;M)
\rightarrow \Gamma^{\infty}(E;M)
$$
by the formula
$$
\nabla^{\lambda}_{X}(\mu)=\{\lambda(X),\mu\}.
$$
Then
$$
\begin{array}{ll}
R^{\nabla^{\lambda}}(X,Y)(\mu)=[\nabla^{\lambda}_{X},\nabla^{\lambda}_{Y}](\mu)-\nabla^{\lambda}_{[X,Y]}(\mu)=\\
~~~~~~~~~~~~~~~~~~~=\{[\lambda(X),\lambda(Y)]-\lambda([X,Y]),\mu\}=0
\end{array}
$$
for arbitrary $X,Y\in\rX^{1}(M),\mu\in\Gamma(E;M)$ since
$a([\lambda(X),\lambda(Y)]-\lambda([X,Y]))=0$ that is
$[\lambda(X),\lambda(Y)]-\lambda([X,Y])\in \Gamma(E;M)$
 and the structural Lie
algebra is commutative.

Conversely. If $E$ is flat, there is a flat connection
$\nabla$ on $E$ which also is a representation of the Lie algebroids
$$
\nabla:TM\rightarrow \D(E),
$$
that is $R^{\nabla}(X,Y)=0$.

By definition of obstruction class this means that $Obs(\nabla)=0\in \mathcal{H}^{3}(M,\nabla,E)$. Then there exist Lie algebroid extensions.
\end{proof}
\newpage
\section{Characteristic Classes}
In this section a system of characteristic classes of transitive Lie
algebroid with commutative adjoint bundle will be described. Then
they will be compared with characteristic classes derived
from Chern-Weil homomorphism by J.Kubarski (\cite{Kub-91e}). 
As a matter of fact we show that the Chern-Weil homomorphism 
does not cover all characteristic classes from categorical 
point of view.

\subsection{A system of characteristic classes for commutative case}
Let $\h$ be a finite dimensional commutative Lie algebra.
Let $\Aut(\h)_{d}$ be the group $\Aut(\h)$ with discrete topology.
The functor $Vector_{d}^{\h}(\bullet)$ associates with each
paracompact topology space $X$ the set $Vector_{d}^{\h}(X)$ of
all vector bundle with structural group $\Aut(\h)_{d}$.
Let $E^{\infty}\rightarrow B_{\Aut(\h)_{d}}$ be
universal bundle with group $\Aut(\h)_{d}$ and let $B_{\Aut(\h)_{d}}$
be the classifying space.

\begin{lemma}(See \cite{Dale-1993}, Definition 11.1, Theorem 11.2, Theorem 12.2)\label{lemma 3.1.1}\label{final space}
There is a bijection
between $Vector_{d}^{\h}(X)$ and the homotopy classes
of continuous maps $[X;B_{\Aut(\h)_{d}}]$.
\end{lemma}

Let $M$ be a smooth manifold and
\begin{equation}
0\xrightarrow{} E \xrightarrow{j} A \xrightarrow{a}TM \xrightarrow{} 0
\end{equation}
be a transitive Lie algebroid with fixed structural commutative Lie
algebra $\h=R^{n}$. Let $\lambda:TM\rightarrow A$ be a splitting.
Define $\nabla=\nabla^{\lambda}$ by a formula
$\nabla^{\lambda}_{X}(\mu)=\{\lambda(X),\mu\}$.
The bundle  $E$ possesses a flat structure
$E^{\nabla}\in Vector_{d}^{\h}(M)$. Let $f:M'\rightarrow M$
be a smooth map and $f^{!!}A$ be the pullback of Lie algebroid $A$ over $f$, that is
\begin{equation}\label{PullbackTLA}
0\xrightarrow{} f^{*}E \xrightarrow{j'} f^{!!}A \xrightarrow{a'}TM' \xrightarrow{} 0.
\end{equation}
Let $\lambda':TM'\rightarrow f^{!!}A$ be a splitting.
Define $\nabla^{'}=\nabla^{\lambda'}$ on $f^{*}E$ and $f^{*}E$ is
corresponding to $(f^{*}E)^{\nabla'}$.

\begin{lemma}\label{Correctness}
\begin{enumerate}\renewcommand{\labelenumi}{$($\roman{enumi}$)$}
\item $\nabla$ and $\nabla'$ are independent of the choice
of $\lambda$ and $\lambda'$,
\item The bundle $(f^{*}E)^{\nabla'}$ is the pullback of $E^{\nabla}$ over $f:M'\rightarrow M$ in the category of vector bundle with discrete structural group $\Aut(\h)_{d}$.
\end{enumerate}
\end{lemma}
\begin{proof}
Statement $(i)$ is obvious.

$(ii):$ Consider the splitting of transitive Lie algebroid (\ref{PullbackTLA})
$$
\lambda':TM'\rightarrow f^{!!}A
$$
by the formula
$$
\lambda'(X')=(X',\lambda(Tf(X'))),
$$
$X'\in TM'$.

Let $\sum\limits_{i}h_{i}\cdot(\mu_{i}\circ f)\in\Gamma(f^{*}E;M)$, here $h_{i}\in C^{\infty}(M'),\mu_{i}\in\Gamma(E;M)$. Then
\begin{equation}
\nabla^{\lambda'}_{X'}(\sum\limits_{i}h_{i}\cdot(\mu_{i}\circ f))=\sum\limits_{i}X'(h_{i})\cdot(\mu_{i}\circ f)+\sum\limits_{i}h_{i}\cdot(\nabla^{\lambda}_{Tf(X')}(\mu_{i})\circ f)
\end{equation}

As $\nabla^{\lambda}$ is flat connection,  there exist chart
$\{\varphi_{\alpha}:E_{U_{\alpha}}\rightarrow
U_{\alpha}\times \h \}_{\alpha\in\Delta}$  which satisfies the condition
\begin{equation}
\varphi_{\alpha}(\nabla^{\lambda}_{X}(\mu_{\alpha}))=
X(\varphi_{\alpha}(\mu_{\alpha}))
\end{equation}
for arbitrary $\mu_{\alpha}\in\Gamma(E_{U_{\alpha}};U_{\alpha})$, $X\in\rX^{1}(M)$.

Consider $\mu\in\Gamma(E_{U_{\alpha}\cap U_{\beta}};U_{\alpha}\cap U_{\beta})$. Then
$$
X(\varphi_{\beta}\circ\varphi_{\alpha}^{-1}\circ \varphi_{\alpha}(\mu))=
X(\varphi_{\beta}(\mu))=\varphi_{\beta}(\nabla^{\lambda}_{X}(\mu))=
\varphi_{\beta}\circ\varphi_{\alpha}^{-1}(\varphi_{\alpha}(\nabla^{\lambda}_{X}(\mu))).
$$
Then
\begin{equation}
X(\varphi_{\beta}\circ\varphi_{\alpha}^{-1}\circ\varphi_{\alpha}(\mu))=
\varphi_{\beta}\circ\varphi_{\alpha}^{-1}(X(\varphi_{\alpha}(\mu)))
\end{equation}

Thus the transition functions
$\{\varphi_{\alpha\beta}\}_{\alpha,\beta\in\Delta}$
are all locally constant.

Let $\{V_{\alpha}^{'}=f^{-1}(U_{\alpha})\}_{\alpha\in\Delta}$
be atlas of charts on $M'$.
Define the homomorphism of $C^{\infty}(V_{\alpha}')-$modules
$$
\psi_{\alpha}:\Gamma(f^{*}E|_{V_{\alpha}'};V_{\alpha}')\rightarrow \Gamma(V_{\alpha}'\times \h;V_{\alpha}')
$$
defined by the formula
$$
\psi_{\alpha}(h_{\alpha,i}\cdot(\mu_{\alpha}^{i}\circ f))=
h_{\alpha,i}\cdot\varphi_{\alpha}(\mu_{\alpha}^{i})\circ f
$$
for $h_{\alpha,i}\cdot(\mu_{\alpha}^{i}\circ f)
\in\Gamma(f^{*}E|_{V_{\alpha}'};V_{\alpha}'),$
where $h_{\alpha,i}\in C^{\infty}(V_{\alpha}'),
\mu_{\alpha}^{i}\in\Gamma(E_{U_{\alpha}};U_{\alpha})$.

As $\varphi_{\alpha}$ is vector bundle isomorphism,
$\psi_{\alpha}$ induces a vector bundle isomorphism.
Then $\{V^{'}_{\alpha},\psi_{\alpha}:f^{*}E|_{V_{\alpha}'}
\rightarrow V_{\alpha}'\times \h\}_{\alpha\in\Delta}$
is a chart for $f^{*}E$. Consider a vector field  $X'\in\rX^{1}(M')$.
Then
$$
\begin{array}{ll}
\psi_{\alpha}(\nabla^{\lambda'}_{X'}(h_{\alpha,i}\cdot(\mu_{\alpha}^{i}\circ f)))\\
=\psi_{\alpha}(X'(h_{\alpha,i})\cdot (\mu_{\alpha}^{i}\circ f)+h_{\alpha,i}\cdot(\nabla^{\lambda}_{Tf(X')}(\mu_{\alpha}^{i})\circ f))=\\
=X'(h_{\alpha,i})\cdot (\varphi_{\alpha}(\mu_{\alpha}^{i})\circ f)+h_{\alpha,i}\cdot (Tf(X')(\varphi_{\alpha}(\mu_{\alpha}^{i}))\circ f)\\
=X'(h_{\alpha,i}\cdot (\varphi_{\alpha}(\mu_{\alpha}^{i}))\circ f)=X'(\psi_{\alpha}(h_{\alpha,i}\cdot(\mu_{\alpha}^{i}\circ f))))
\end{array}
$$

The transition functions
$$
\psi_{\alpha\beta}:V_{\alpha}'\cap V_{\beta}'\rightarrow \Aut(\h)_{d}
$$
are defined by
$$
\psi_{\alpha\beta}(x')=\varphi_{\alpha\beta}(f(x'))
$$
for $x'\in V_{\alpha}\cap V_{\beta}$.

So $(f^{*}E)^{\nabla'}$ is the pullback of $E^{\nabla}$ over $f:M'\rightarrow M$ in the category of vector bundle with discrete structural group $\Aut(\h)_{d}$.
\end{proof}
The Lemma \ref{Correctness} shows that the following definition is corrected.

\begin{definition}\label{CharClass}
Let $\h$ be a commutative Lie algebra and $M$ be a smooth manifold.
Let $A\in \mathcal{TLA}_{\h}(M)$, with splitting $\lambda$.
Let $E^{\nabla^{\lambda}}$ be the correspondent Lie algebra bundle
with flat structure.
Let $\theta:Vector_{d}^{\h}(M)\rightarrow [M;B_{\Aut(\h)_{d}}]$
be the bijection defined in Lemma \ref{final space}.
Then $\theta(E^{\nabla^{\lambda}})=[f]\in [M;B_{\Aut(\h)_{d}}]$
induces a homomorphism
$$
f^{*}:H^{*}(B_{\Aut(\h)_{d}};R)\rightarrow H^{*}(M;R).
$$
The class $f^{*}(c)\in H^{*}(M;R)$ is characteristic class of $A$,
for arbitrary $c\in H^{*}(B_{\Aut(\h)_{d}};R)$.
\end{definition}

\subsection{Chern-Weil homomorphism}

\begin{definition}(see \cite{Kub-91e}, page17)
Given a transitive Lie algebroid (A,q,M,\{,\},a) with adjoint bundle $L$.
The adjoint representation of a transitive Lie algebroid $A$ is
$$
ad:A\rightarrow \cD(L)
$$
defined by
$$
ad(\xi)(\nu)=\{\xi,\nu\}
$$
for $\xi\in\Gamma(A;M),\quad \nu\in\Gamma(L;M)$. Let $L^{*}$ be dual bundle of $L$ and $\bigvee^{k}L^{*}$ is $k-th$ symmetric power of $L^{*}$(see \cite{Greub-1967}, page 191). The adjoint representation $ad$ can rise to
$$\bigvee^{k}ad^{\natural}:A\rightarrow \cD(\bigvee^{k}L^{*})$$
such that
$$
\begin{array}{l}
<\bigvee^{k}ad^{\natural}(\xi)(\varphi),
\nu^{1}\vee\nu^{2}\vee...\vee\nu^{k}> = \\
= a(\xi)(<\varphi,\nu^{1}\vee\nu^{2}\vee...
\vee\nu^{k}>)-\sum\limits_{i=1}^{k}<\varphi,\nu^{1}\vee...
\vee\{\xi,\nu^{i}\}\vee...\vee\nu^{k}>
\end{array}
$$
for $\xi\in\Gamma(A;M),\varphi\in\Gamma(\bigvee^{k}L^{*};M), \nu^{i}\in\Gamma(L;M)$.
\end{definition}
\begin{remark}
Here we only consider the vector bundle structure of $L$ that is
commutative Lie algebra structure.
Hence we use notation $\D(L)$ and $\D(\bigvee^{k}L^{*})$.
\end{remark}

\begin{definition}(see \cite{Kub-91e}, Definition 2.3.1)
Given an arbitrary transitive Lie algebroid
$ 0\xrightarrow{} L\xrightarrow{} A\xrightarrow{a} TM\xrightarrow{} 0$.
Let $L^{*}$ be dual bundle of $L$.
A section $\varphi\in\Gamma(\bigvee^{k}L^{*};M)$
is called $\bigvee^{k}ad^{\natural}-$invariant
if $\bigvee^{k}ad^{\natural}(\xi)(\varphi)=0$
for all $\xi\in\Gamma(A;M)$. The space of all
$\bigvee^{k}ad^{\natural}-$invariant sections of
$\bigvee^{k}L^{*}$ is denoted by $\Gamma^{I}(\bigvee^{k}L^{*};M)$.

\end{definition}
\begin{definition}[\bf Chern-Weil homomorphism]
(see \cite{Kub-91e}, page 29)
Given a transitive Lie algebroid (A,q,M,\{,\},a)
with adjoint bundle $L$.
Let $\lambda:TM\rightarrow A$ be a splitting and
$R^{\lambda}\in\Omega^{2}(M;L)$ be the curvature tensor,
$R^{\lambda}(X,Y)=\{\lambda(X),\lambda(Y)\}-\lambda([X,Y])$.

Define a homomorphism of $C^{\infty}(M)-$modules
$$
\chi_{(A,\lambda),I}:\Gamma^{I}(\bigvee^{k}L^{*};M)\rightarrow \Omega^{2k}(M)
$$
by the formula
$$
\chi_{(A,\lambda),I}=\frac{1}{k!}<\varphi,R_{\lambda}\vee
R_{\lambda}\vee...\vee R_{\lambda}>
$$
for $\varphi\in\Gamma(\bigvee^{k}L^{*};M)$. Here
$$
\begin{array}{l}
<\varphi,R_{\lambda}\vee...\vee R_{\lambda}>(X_{1},X_{2},...,X_{2k})=\\\\
= <\varphi,\frac{1}{2^{k}}
\sum\limits_{\sigma}(-1)^{\sigma} R_{\lambda}(X_{\sigma(1)},X_{\sigma(2)})\vee R_{\lambda}(X_{\sigma(3)},X_{\sigma(4)})\vee...\vee R_{\lambda}(X_{\sigma(2k-1)},X_{\sigma(2k)})>
\end{array}
$$

The forms from the image of $\chi_{(A,\lambda),I}$ is closed
(see \cite{Kub-91e}, proposition 4.1.2).
Then Chern-Weil homomorphism is defined by the composition
$$
h_{(A,\lambda)}:\bigoplus\limits_{k\geq 0}\Gamma^{I}(\bigvee^{k}L^{*};M)\xrightarrow{\chi_{(A,\lambda),I}} \kernel d^{\nabla^{\lambda}} \xrightarrow{i} H^{*}_{DRam}(M;R).
$$
\end{definition}
The Chern-Weil homomorphism has functorial property and
is independent of the choice of splitting
(see \cite{Kub-91e}, theorem 4.2.2, theorem 4.3.7).
Then $h_{(A,\lambda)}$ can be denoted as
$$h_{A}:\bigoplus\limits_{k\geq 0}\Gamma^{I}(\bigvee^{k}L^{*};M)
\rightarrow H^{*}_{DRam}(M;R).$$

\subsection{Example}

The following example shows that the Chern-Weil homomorphism does not cover
all categorical characteristic classes.

Consider a flat $1$--dimensional vector bundle $E$ over a torus
$T^{2}=S^{1}\times S^{1}$. We will consider $E$ as a Lie algebra bundle
with commutative Lie algebra $\h\approx \R^{1}$. The structural group
of the bundle $E$ is the group $R^{*}=R\setminus\{0\}$
with discrete topology. The flat structure on $E$ is defined by an atlas
of charts $\{U_{\alpha}\}$ with trivialization of the bundle
$E$ on each chart $U_{\alpha}$ such that all transition function
are locally constant. Transition functions are fully defined by a
representation of the fundamental $\pi_{1}(T^{2})$ in the structural
group $\Aut(\h)_{d}$, $\rho:\pi_{1}(T^{2})\rightarrow \Aut(\h)_{d}$.

There is a flat connection $\nabla$ on $E\rightarrow T^{2}$ which
corresponds to the flat structure on $E$. This means that
the connection on each chart $U_{\alpha}$ (after trivialization
of the bundle $E$)
coincides with usual derivative ($\nabla_{X}=\frac{\partial}{\partial X}$).

Construct a Lie algebroid $\mathcal{A}:$
$$
0\rightarrow E\rightarrow T(T^{2})\bigoplus E\rightarrow T(T^{2})\rightarrow 0
$$
with bracket
$$
\{(X,\mu),(Y,\nu)\}=([X,Y],\nabla_{X}(\nu)-\nabla_{Y}(\mu)+\Omega(X,Y))
$$
for $(X,\mu),(Y,\nu)\in\Gamma(T(T^{2})\bigoplus E;T^{2})$.
Here $\Omega\in\kernel d^{\nabla}\subset\Omega^{2}(T^{2},E)$.
Let $E^{*}$ be the bundle dual to $E$.
Let $f\in\Gamma^{I}(E^{*};T^{2})$.
Then
\begin{equation}\label{invsection}
\begin{array}{l}
ad^{\natural}((X,\nu))(f)(\mu)=X(f(\mu))-f(\{X\oplus\nu,0\oplus\mu\})= \\
=X(f(\mu))-f(\nabla_{X}(\mu))=0
\end{array}
\end{equation}
for arbitrary $\mu\in \Gamma(E;T^{2}),\quad (X,\nu)\in\Gamma(T(T^{2})\oplus E;T^{2})$.
Hence locally on the chart $U_{\alpha}$ the function $f$ is constant.

This means that in the case of nontrivial representation 
the space $\Gamma^{I}(E^{*};T^{2})$ has only trivial element. Thus the characteristic class for $\mathcal{A}$ defined by Chern-Weil homomorphism by J.Kubarski is trivial.

On the other hand  the characteristic classes due to
definition \ref{CharClass} are not trivial.
Namely the structural group $\Aut(\h)_{d}$ is isomorphic to
$R\setminus \{0\}\approx\mathbb{Z}_{2}\times R$.

Hence the classifying space  for vector bundle with discrete
structural group $\h$ is  $B_{\mathbb{Z}_{2}}\times B_{R}$.
We have $B_{\mathbb{Z}_{2}}\sim \mathbb{R}\mathbb{P}^{\infty}$.
The group $R$ is a direct sum
$R\approx\bigoplus_{\alpha\in A}\mathbb{Q}_{\alpha}$ where each group
$\mathbb{Q}_{\alpha}$ is isomorphic to rational numbers,
$\mathbb{Q}_{\alpha}\approx\mathbb{Q}$. The group $\mathbb{Q}$ is isomorphic
to the direct limits
$\mathbb{Q}=\lim\limits_{\rightarrow}\left(\mathbb{Z}_{n}, \omega_{n}\right)$,
where all $\mathbb{Z}_{n}$ are isomorphic to $\mathbb{Z}$, and
$\omega_{n}:\mathbb{Z}_{n}\rightarrow \mathbb{Z}_{n+1}, \omega_{n}(k)=(n+1)k$.

Thus the classifying space $B_{R}$ can be represent as a direct limits

$$
B_{R}=\lim\limits_{\stackrel{\rightarrow}{b\subset B}}\mathbb{T}_{b},
$$
where each $b\in B$ is a finite collection of indexes
$$b=\{\alpha_{1},n_{1},\alpha_{2},n_{2},\dots, \alpha_{k}, n_{k}\},
\alpha_{j}\in A, n_{j}\in\mathbb{Z}$$
that are ordered in the natural way,
$\mathbb{T}_{b}=\prod\limits^{k}_{j=1}S^{1}_{\alpha_{j},n_{j}}\approx\mathbb{T}^{k}.$

The cohomology group $H^{*}(B_{\Aut(\h)_{d}};R)$ can be describe 
in the following way:
$$H^{*}(B_{\mathbb{Z}_{2}};R)\approx R;$$  
$$
H^{*}(B_{R};R)\approx \lim\limits_{\stackrel{\leftarrow}{b\subset B}}
H^{*}(\mathbb{T}_{b};R).
$$
The representation $\rho:\pi_{1}(T^{2})\rightarrow \Aut(\h)_{d}$ induces 
the map
$$
B_{\rho}:\mathbb{T}_{2}\rightarrow B_{\mathbb{Z}_{2}}\times B_{R},
$$
and the homomorphism in cohomology
$$
B_{\rho}^{*}:H^{*}(B_{\mathbb{Z}_{2}}\times B_{R};R)\rightarrow
H^{*}(\mathbb{T}_{2};R).
$$

\begin{lemma}[Key lemma]
The homomorphism $B_{\rho}^{*}$ is surjective.
\end{lemma}

The  example show that Chern-Weil homomorphism cannot define 
all characteristic classes for transitive Lie algebroid.

\begin{remark}
This example show that there is a natural problem to generalize 
the Chern-Weil homomorphism for non trivial flat bundle $ZL$ of 
local coefficients
for cohomologies that  contain characteristic classes.
\end{remark}

This work is partly supported by scientific program for the Chief 
International Academic Adviser of 
the Harbin institute of technology (2011-2014)(China)
and 
Russian foundation of Basic research grant No.11-01-00057-a.


\begin{thebibliography}
{ccccc}
\bibitem{Mck-2005}
  {\sf Mackenzie, K.C.H.,}
  {\it General Theory of Lie Groupoids and Lie Algebroids,}
  {\rm Cambridge University Press,}(2005)
\bibitem{Hatcher-2005}
  {\sf Allen Hatcher}
  {\it Vector bundles and K-theory,}(2005)
\bibitem{Kub-91e}
    {\sf Kubarski, J.},
    {\it The Chern-Weil homomorphism of regular Lie algebroids},
    {\rm Publications du Department de Mathematiques,
    Universite Claude Bernard - Lyon-1},
    {(1991) pp.4--63}
\bibitem{Mi-2010}
    {\sf Mishchenko, A.S.,}
    {\it Transitive Lie algebroids - categorical point of view},
    {\rm arXiv:1006.4839v1 [math.AT], 2010}

\bibitem{Mi-2011}
    {\sf Mishchenko, A.S.,}
    {\it Characteristic classes of transitive Lie algebroids.
    Categorical point of view},
    {\rm arXiv:1111.6823v1 [math.AT]},
    {2011.}

\bibitem{Dale-1993}
    {\sf Dale Husemoller},
    {\it Fiber Bundle},
    {\rm Springer-Verlag},
    {(1993).}
\bibitem{Greub-1967}
    {\sf W. H. Greub},
    {\it Multilinear Algebra},
    {\rm Springer-Verlag Berlin Heidelberg New York},
    {(1967)}

\end{thebibliography}
\end{document}